\documentclass[letterpaper, 10pt, conference]{ieeeconf}
\IEEEoverridecommandlockouts
\usepackage{amsmath,amssymb}
\usepackage{graphicx}
\usepackage{color,xspace,enumerate}

\newcommand{\utwi}[1]{\mbox{\boldmath $#1$}}

\renewcommand{\hat}{\widehat}
\renewcommand{\tilde}{\widetilde}

\newcommand{\cD}{{\cal D}}

\newcommand{\cN}{{\cal N}}

\newcommand{\cG}{{\cal G}}
\newcommand{\cA}{{\cal A}}

\newcommand{\cX}{{\cal X}}

\newcommand{\bc}{{\bf c}}

\newcommand{\bs}{{\bf s}}
\newcommand{\bx}{{\bf x}}
\newcommand{\bu}{{\bf u}}
\newcommand{\bv}{{\bf v}}
\newcommand{\bw}{{\bf w}}

\newcommand{\bz}{{\bf z}}
\newcommand{\by}{{\bf y}}

\newcommand{\bC}{{\bf C}}

\newcommand{\bPhi}{{\utwi{\Phi}}}

\newcommand{\reals}{\mathbb{R}}

\newcommand{\sfT}{\textsf{T}}

\newcommand{\proj}{\mathsf{Proj}}

\usepackage{epstopdf}

\newtheorem{lemma}{Lemma}
\newtheorem{theorem}{Theorem}
\newtheorem{corollary}{Corollary}

\newtheorem{assumption}{Assumption}
\newtheorem{remark}{Remark}

\title{\LARGE \bf Asynchronous and Distributed Tracking \\ of Time-Varying Fixed Points} 

\author{Andrey Bernstein and Emiliano Dall'Anese 
\thanks{A. Bernstein is with the National Renewable Energy Laboratory (NREL), Golden, CO, USA; email: andrey.bernstein@nrel.gov. E. Dall'Anese is with the University of Colorado Boulder, Boulder, CO, USA; email: emiliano.dallanese@colorado.edu. This work was supported by the Laboratory Directed Research and Development Program at NREL.  NREL is a national laboratory of the U.S. Department of Energy Office of Energy Efficiency and Renewable Energy operated by the Alliance for Sustainable Energy, LLC. }
}

\begin{document}
\maketitle

\begin{abstract}
This paper develops an algorithmic framework for tracking fixed points of time-varying contraction mappings. Analytical results for the tracking error are established for the cases where: (i)~the underlying contraction self-map changes at each step of the algorithm; (ii) only an imperfect information of the map is available; and, (iii) the algorithm is implemented in a distributed fashion, with communication delays and packet drops leading to asynchronous algorithmic updates. The analytical results are applicable to several classes of problems, including time-varying contraction mappings emerging from online and asynchronous implementations of gradient-based methods for time-varying convex programs. In this domain, the proposed framework can also capture the operating principles of feedback-based online algorithms, where the online gradient steps are suitably modified to accommodate actionable feedback from an underlying physical or logical network. Examples of applications and illustrative numerical results are provided. 
\end{abstract}

\section{Introduction}
\label{sec:introduction}

A number of iterative algorithms can be expressed in the form~\cite{Bertsekas1999} 
\begin{equation} 
\label{eqn:fixed_point}
\bx^{(k+1)} = f (\bx^{(k)} ), 
\end{equation}
where $f: \mathbb{R}^m \rightarrow \mathbb{R}^m$ is a continuous map, $\bx^{(k)}$ is an $m$-dimensional real vector, and $k \in \mathbb{N}$ denotes the iteration index. Assuming the existence of a (possibly unique) \emph{fixed point} $\bx^*$ that satisfies the equation $\bx^* = f(\bx^*)$,
convergence characteristics of the iterative method~\eqref{eqn:fixed_point} to $\bx^*$ hinge on pertinent regularity conditions for $f(\cdot)$; typical conditions include a contraction map~\cite{Bertsekas1999}, an $\alpha$-averaged operator \cite{Simonetto17}, and a paracontraction map~\cite{Elsner92}. When~\eqref{eqn:fixed_point} is implemented in a distributed setting, with a number of agents sharing the computational burden of the algorithmic update and coordinating through information exchange~\cite{Bertsekas_fixedpoint}, additional prerequisites are imposed on the communication architecture (e.g., connectivity of the possibly time-varying communication graph, as well as asynchronicity of the information exchange)~\cite{Liu17,Fullmer16,Fullmer17} or the number of communication losses~\cite{opfpursut_comm}. Examples of instances of~\eqref{eqn:fixed_point} include (projected) gradient-based methods for solving constrained optimization problems~\cite{Bertsekas_ConvexAnalysis}, iterative methods for solving linear and nonlinear systems of equations~\cite{Mou15}, and consensus problems for networked systems~\cite{OlfatiSaber07,Fang08} just to mention a few. 

Consider now a sequence of \emph{time-varying} mappings $f^{(t)}: \mathbb{R}^m \rightarrow \mathbb{R}^m$, where $t \in \mathbb{N}$ is a \emph{temporal} index, along with the sequence of $m$-dimensional vectors of fixed points
\begin{align} \label{eqn:fixed_point_time_varying}
\bx^{(\ast , t)} = f^{(t)}(\bx^{(\ast , t)}) \, , t \in \mathbb{N} . 
\end{align}
One way to identify the sequence of  fixed points~\eqref{eqn:fixed_point_time_varying} is to sequentially compute~\eqref{eqn:fixed_point} until convergence within each interval $[t, t+1)$; that is, 
at time $t$ one can perform the following computation until convergence:  
\begin{equation} 
\label{eqn:fixed_pointint}
\bx^{(t, k+1)} = f^{(t)} (\bx^{(t, k)} ), \, k = 1, 2, \ldots
\end{equation}
and subsequently advance to the next time step $t \rightarrow t+1$. In lieu of the offline strategy~\eqref{eqn:fixed_pointint}, this paper considers the following inexact \emph{online} (or \emph{running}) algorithm
\begin{equation} 
\label{eqn:fp_algo_online}
\bx^{(t+1)} = \tilde{f}^{(t)} (\bx^{(t)}) \, ,
\end{equation}
where: (c1) the map may vary at each step of the algorithm; and (c2) $\tilde{f}^{(t)}$ represents an  \emph{approximation} of the true map $f^{(t)}$. The term ``approximation'' will be supported by a proper analytical definition, and examples of approximate maps will be given shortly. The running algorithm~\eqref{eqn:fp_algo_online} is motivated by real-time implementations, where it might not be affordable to run the batch iterations~\eqref{eqn:fixed_pointint} within each time interval due to underlying computational and communication constraints.  

This paper analyzes the tracking properties of~\eqref{eqn:fp_algo_online}, defined in terms of the distance between $\bx^{(t)}$ and $\bx^{(\ast , t)}$ relative to a given norm, under the operational setting (c1) and (c2), and when (c3) the iterations are implemented in a \emph{distributed} and \emph{asynchronous} fashion. A distributed setting implies that the evaluation of the map can be decoupled across various agents (or processors), sharing information on local portions of the vector $\bx^{(t)}$ via an underlying communication network~\cite{Bertsekas_fixedpoint}. Asynchronicity emerges from  latencies in the end-to-end communication paths between the nodes as well as communication drops; this setting can be also modeled under the formalism of time-varying (communication) graphs~\cite{Liu17,Fullmer16}. Under suitable conditions, the paper establishes analytical results for the tracking error in the implementation settings (c1)--(c3).      
When the mapping $f^{(t)}$ is known and the iterations are performed in a synchronous manner, the convergence properties of the running synchronous algorithm~\eqref{eqn:fp_algo_online} were first investigated in~\cite{Simonetto17} under the assumption that $f^{(t)}$ is an $\alpha$-averaged operator. In this setting,~\cite{Simonetto17} showed that \eqref{eqn:fp_algo_online} is close in spirit to online implementations of the Mann-Krasnosel'skii iterations~\cite{Krasno,Mann,Moreau}; see~\cite{Simonetto17} and pertinent references therein for a  survey of results related to convergence of batch and online versions of the Mann-Krasnosel'skii method. 

Overall, the main contribution of this paper is two-fold. 

\noindent (i) The paper extends the results of~\cite{Bertsekas_fixedpoint, Simonetto17} (and pertinent references therein) to the cases (c1)--(c3). Compared to~\cite{Simonetto17}, however, we analyze a more restricted case of contraction mappings (which is a special case of the $\alpha$-averaged operators considered in \cite{Simonetto17}). To the best of our knowledge, our paper is the first to address  online, distributed, and asynchronous implementations of the fixed-point iterations, with a possibly imperfect knowledge of the mappings.   An asynchronous implementation of iteration \eqref{eqn:fp_algo_online} was considered in \cite{Frommer00}; however, the analysis of \cite{Frommer00} is limited to  the assumption of a \emph{common fixed point} at every time step and a perfect knowledge of the mappings.

\noindent (ii)  As pointed out in~\cite{Bertsekas_fixedpoint}, in the asynchronous case, the convergence results depend on the norm used to define the contraction. When the $\ell_\infty$ norm is utilized, it was shown that the asynchronous implementation of  \eqref{eqn:fixed_point} converges to the fixed point~\cite{Bertsekas_fixedpoint}; however, this is not necessarily the case for other norms. This paper provides further conditions on the family of maps that ensure  convergence of the asynchronous iteration when the map is a contraction with respect to the $\ell_2$ norm. This permeate benefits in optimization-related applications, as maps  are contractions in $\ell_2$ norm (but not necessarily in $\ell_\infty$).   

Finally, to better appreciate the analysis of the algorithm under imperfect knowledge of the mappings, we show that the proposed framework can be utilized to model and analyze  \emph{feedback-based online algorithms} for solving 
time-varying optimization problems~\cite{opfPursuit,AndreayOnlineOpt,Bolognani_feedback_15,Hauswirth16, Tang17}; in this domain, gradient steps of a primal-dual-type method are suitably modified to accommodate actionable feedback from an underlying physical or logical network. Consider, as an illustrative example (a more detailed explanation will be provided shortly in the paper), the following projected gradient step associated with the time-varying problem $\min_{\bx \in \cX^{(t)}} g^{(t)}(\by^{(t)}(\bx))$: 
\begin{align}
\hspace{-.2cm} \bx^{(t+1)} = \proj_{\cX^{(t)}} \Big\{ \bx^{(t)} - \alpha \bC^\sfT \nabla_\by g^{(t)} (\bC \bx^{(t)} + \bw^{(t)})  \Big\} \label{eq:online}
\end{align}
where $\cX(t)$ is a time-varying convex set, $\proj_{\cX}(\bz) := \arg \min_{\bx \in \cX} \|\bz - \bx\|_2$ denotes projection onto a convex set, $\alpha > 0$ is the step size, and $\by^{(t)}(\bx) := \bC \bx^{(t)} + \bw^{(t)}$ is a model for some measurable quantities in a physical or logical network~\cite{opfPursuit,AndreayOnlineOpt,LowOnlineOPF,Hauswirth16}. A solution of the problem at time $t$ is in fact a fixed point $\bx^{(\ast , t)} = \proj_{\cX^{(t)}} \{ \bx^{(\ast , t)} - \alpha \bC^\sfT \nabla_\by g^{(t)} (\bC \bx^{(\ast , t)} + \bw^{(t)}) \}$.  When a measurement $\tilde{\by}^{(t)}$ of $\by^{(t)}(\bx)$ is available, a feedback-based counterpart of~\eqref{eq:online} amounts to the following iteration~\cite{opfPursuit}   
\begin{align}
\hspace{-.2cm} \bx^{(t+1)} = \proj_{\cX^{(t)}} \Big\{ \bx^{(t)} - \alpha \bC^\sfT \nabla_\by g^{(t)} (\tilde{\by}^{(t)})  \Big\} \label{eq:online_meas}
\end{align}
Since $\tilde{\by}^{(t)}$ is a noisy version of $\by^{(t)}(\bx)$ due to, e.g., measurement noise, model mismatches, or even missing measurements, step~\eqref{eq:online_meas} is based on an approximation of the (true) mapping $f^{(t)}(\bx^{(t)}) = \proj_{\cX^{(t)}} \{ \bx^{(t)} - \alpha \bC^\sfT \nabla_\by g^{(t)} (\bC \bx^{(t)} + \bw^{(t)}) \}$.    

The rest of the paper is organized as follows. Section~\ref{sec:tracking} will address the tracking of time-varying fixed points, and Section~\ref{sec:asynchronous} will analyze the tracking performance of an asynchronous and distributed algorithm. Section~\ref{sec:examples} will outline two examples of applications. Section~\ref{sec:results} will provides illustrative numerical results. Section~\ref{sec:conc} concludes the paper. 

\section{Tracking of Time-varying Fixed points}
\label{sec:tracking}
Consider a sequence of mappings $f^{(t)}: \reals^m \rightarrow \reals^{m}$, where $t \in \mathbb{N}$ a temporal index (for a given discretization of the temporal axis). Assume that $f^{(t)}: \reals^m \rightarrow \reals^{m}$ is a self map and a contraction. Formally, the following assumptions are presupposed.

\begin{assumption} \label{asm:contr}
There exists a closed set $\cD \subseteq \reals^m$ such that for each $t \in \mathbb{N}$:
\begin{enumerate}[(i)]
\item For all $\bx \in \cD$, $f^{(t)}(\bx) \in \cD$; and,
\item There exists a scalar $0 < L^{(t)} < 1$ such that $\| f^{(t)}(\bx) - f^{(t)}(\bx') \| \leq L^{(t)} \|\bx - \bx' \|$  for all $\bx, \bx' \in \cD$.
\end{enumerate}
Furthermore, $L := \sup_{t \geq 1} L^{(t)} < 1$.
\end{assumption}

Under Assumption \ref{asm:contr}, for each time step $t$, Banach fixed point theorem asserts that there exists a unique solution to the equation
\begin{equation} \label{eqn:fp}
\bx = f^{(t)} (\bx) \, ,
\end{equation}
which is denoted by $\bx^{(\ast,t)} \in \cD$. The rate of change of the fixed points is characterized by the following bound with respect to a given norm:
\begin{equation} \label{eqn:sigma}
\sigma^{(t)} := \|\bx^{(*, t+1)} - \bx^{(*, t)} \| \, ,
\end{equation}
and the following assumption is imposed. 

\begin{assumption} \label{asm:sigma}
We have that $\sigma := \sup_{t \geq 1} \sigma^{(t)} < \infty$.
\end{assumption}

This paper seeks the development of an algorithmic solution to track the $m$-dimensional vectors of fixed points $\bx^{(\ast,t)}$, $t \in \mathbb{N}$, when: (c1)~the underlying contraction self-map changes during the execution of the algorithm, leading to a \emph{running} or \emph{online} computation of the sequence of fixed points; and (c2) only an \emph{imperfect} information of the map is available. As discussed in Section~\ref{sec:examples} below, (c1) addmairesses the case where it is not feasible to solve \eqref{eqn:fp} to convergence within each time step due to complexity of evaluating $f^{(t)}$ iteratively; on the other hand,~(c2) copes with cases where the map $f^{(t)}$ is either not available or its exact evaluation is computationally infeasible. In particular, this setting will include the analysis of  \emph{feedback-based online algorithms}, which will be explained in Section~\ref{sec:examples}. 

Consider then an \emph{approximate mapping} $\tilde{f}^{(t)} :\cD \rightarrow \cD$, which is related to $f^{(t)}$ through the following assumption.

\begin{assumption}	\label{asm:approx}
For each $t$, there exists $e_f^{(t)} < \infty$ such that
\begin{equation}
\|f^{(t)}(\bx) - \tilde{f}^{(t)}(\bx)\| \leq e_f^{(t)}, \, \forall \bx \in \cD,
\end{equation}
and $e_f := \sup_{t \geq 1} e_f^{(t)} < \infty$.
\end{assumption}

\begin{remark}
In some applications, the set $\cD$ can be the entire Euclidean space $\reals^m$. In particular, this is true when considering the gradient-based optimization methods; see Section \ref{sec:feedback} below for details. Note, however, that our framework also captures a more general case, where $\cD$ can be a proper subset of $\reals^m$. This is the case, for example, of the load-flow application considered in Section \ref{sec:loadflow} below.
\end{remark}

The following online (i.e., running) algorithm is then proposed to track the sequence of the fixed points $\{\bx^{(*, t)} \}$:
\begin{equation} \label{eqn:fp_algo}
\bx^{(t+1)} = \tilde{f}^{(t)} (\bx^{(t)} ), \, t \in \mathbb{N},
\end{equation}
for some initial point $\bx^{(1)} \in \cD$. The convergence properties of the iterative algorithm defined by \eqref{eqn:fp_algo} are analyzed next.

\begin{theorem}
\label{thm:result_fixedpoint}
Under Assumptions \ref{asm:contr}, \ref{asm:sigma}, and \ref{asm:approx}, it holds that
\begin{align}
\|\bx^{(t+1)} - \bx^{(*, t+1)} \| \leq & \,\, \beta^{(t,0)} \|\bx^{(1)} - \bx^{(*, 1)}\| \nonumber \\
& + \sum_{\tau = 1}^t \beta^{(t, \tau)} \left(  e_f^{(\tau)} + \sigma^{(\tau)} \right)  \label{eq:bound_iter} 
\end{align}
for each $t$, where
\begin{equation} \label{eqn:beta}
\beta^{(t, \tau)} := 
\begin{cases}
\prod_{\ell = \tau + 1}^t L^{(\ell)}, & \textrm{if~} \tau = 0, \ldots, t - 1 \\
1, &  \textrm{if~} \tau = t.
\end{cases}
\end{equation}
In particular, the following  asymptotic bound holds: 
\begin{equation}
\label{eq:result_fixedpoint}
\limsup_{t \rightarrow \infty} \|\bx^{(t)} - \bx^{(*, t)} \| \leq \frac{e_f + \sigma}{1 - L}.
\end{equation}
\end{theorem}

\vspace{.2cm}

\begin{proof}
Consider starting with the following inequalities
\begin{align}
&\|\bx^{(t)} - \bx^{(*, t-1)} \| = \left\|\tilde{f}^{(t-1)} \left(\bx^{(t-1)} \right) - f^{(t-1)} \left(\bx^{(*, t-1)} \right) \right \| \nonumber \\
& \hspace{.9cm} \leq \left\|\tilde{f}^{(t-1)} \left(\bx^{(t-1)} \right) - f^{(t-1)} \left(\bx^{(t-1)} \right) \right\| \nonumber \\
& \hspace{1.3cm}  + \left\|f^{(t-1)} \left(\bx^{(t-1)} \right) - f^{(t-1)} \left(\bx^{(*,t-1)} \right) \right\| \nonumber \\
& \hspace{.9cm} \leq e_f^{(t-1)} + L^{(t-1)} \|\bx^{(t-1)} -  \bx^{(*,t-1)}\|, \label{eqn:inexact_bound}
\end{align}
where the first inequality follows by the triangle inequality and the second  is due to Assumption \ref{asm:contr} (ii) and Assumption \ref{asm:approx}. From~\eqref{eqn:sigma}, and by the triangle inequality, one has that: 
\begin{align}
&\|\bx^{(t+1)} - \bx^{(*, t+1)} \| \nonumber \\
& = \|\bx^{(t+1)} - \bx^{(*, t)} + (\bx^{(*, t)}- \bx^{(*, t+1)}) \| \nonumber \\
& \leq \|\bx^{(t+1)} - \bx^{(*, t)} \| + \|\bx^{(*, t+1)} - \bx^{(*, t)} \| \nonumber \\
& \leq e_f^{(t)} + \sigma^{(t)} + L^{(t)} \|\bx^{(t)} -  \bx^{(*,t)}\|.
\end{align}
Applying the last inequality recursively, one obtains~\eqref{eq:bound_iter}. 
where $\beta^{(t, \tau)}$ is given in \eqref{eqn:beta}.
Recall that $\beta^{(t, \tau)} \leq L^\tau$. Then, leveraging the definition of $L$ and taking the $\limsup$, the result~\eqref{eq:result_fixedpoint} follows.
\end{proof}

\begin{remark}
If a perfect knowledge of the maps is available,~\eqref{eq:result_fixedpoint} boils down to results similar to~\cite{Simonetto17} for the running Mann-Krasnosel'skii method. If, in addition, $\sigma = 0$, then we recover classic results for \emph{static} fixed points. 
\end{remark}

The result~\eqref{eq:result_fixedpoint} will be revisited in the ensuing section for the case of asynchronous and distributed implementations.  

\section{Asynchronous and Distributed Algorithm}
\label{sec:asynchronous}

\subsection{Distributed Computation}


Consider a network of agents and let $\cG = (\cN, \cA)$ be a \emph{dependency graph} where $\cN := \{1, \ldots, N\}$ is the set of agents and $\cA$ is the set of directed edges, which represent information exchanges that are required in order to perform iteration \eqref{eqn:fp_algo}~\cite{BeT89}. In particular, each agent $i$ updates a portion $\bx_i^{(t+1)} \in \mathbb{R}^{m_i}$ of the vector $\bx^{(t+1)}$ based on a (local) map $\tilde{f}^{(t)}_i$, the current sub-vector $\bx_i^{(t)}$, and the sub-vectors of some other agents. Therefore, a directed edge $(j, i) \in \cA$ exists if the function $\tilde{f}^{(t)}_i$ depends on $\bx_j$ for all $t$. Notice that $\tilde{f}^{(t)} = [(\tilde{f}^{(t)}_1)^\sfT, \ldots, (\tilde{f}^{(t)}_N)^\sfT]^\sfT$ and $\sum_{i = 1}^N m_i = m$ for consistency. Similar to the previous section, each noise-free function $f^{(t)}_i$ is assumed Lipshitz continuous with a given constant $L_i > 0$. Accordingly, let
\begin{equation}
\cN_i := \left \{j \in \cN : \, (j, i) \in \cA \right\}
\end{equation}
denote the set of nodes that are neighbors of node $i$. With these definitions, the iteration \eqref{eqn:fp_algo} can be equivalently written as
\begin{equation} \label{eqn:fp_algo_dist}
\bx^{(t+1)}_i = \tilde{f}^{(t)}_i \left(\{\bx^{(t)}_j\}_{j \in \cN_i \cup \{ i\}} \right), \quad  i \in \cN.
\end{equation}
In this case, the computation can be distributed across agents, provided that the variables from the neighboring nodes are shared. This paper focuses 
on the case where the communication structure does not change with time (i.e., the dependency graph  $\cG$ is not time varying), although extensions to time-varying dependency graphs will be explored in future research efforts.

In case of ideal communications,~\eqref{eqn:fp_algo_dist} inherits the convergence properties of Theorem~\ref{thm:result_fixedpoint}. An asynchronous implementation is addressed next.

\subsection{Asynchronous Computation}
The following setting is  considered:
\begin{itemize}
\item At each node $i$, the variables $\{\bx^{(t)}_j\}_{j \in \cN_i}$ of the neighboring agents might be \emph{delayed} due to communication constraints or might even be lost due to packets drops;
\item The computational time across agents is similar, and it is negligible relative to communication latencies. 
\end{itemize}
In this setting,  the algorithmic update~\eqref{eqn:fp_algo_dist} is re-written as:
\begin{equation} \label{eqn:fp_algo_dist_async}
\bx^{(t+1)}_i = \tilde{f}^{(t)}_i \left(\bx^{(t)}_i, \{\tilde{\bx}^{(t)}_j\}_{j \in \cN_i} \right), \quad i \in \cN,
\end{equation}
where
\begin{equation}
\tilde{\bx}^{(t)}_j = \bx^{(D^{(t)}_{i, j})}_j
\end{equation}
for some $D^{(t)}_{i, j} \in \{1, \ldots, t\}$. If $D^{(t)}_{i, j} < t$, then $\tilde{\bx}^{(t)}_j$ is an outdated copy of the variable associated with agent $j$. 

\begin{remark}
The considered setting does not capture the case where  agents exhibit different computational capabilities.  The analysis of a more general case where asynchronous updates are due to both communication constraints and different computational times is the subject of ongoing work. 

\end{remark}

Convergence of the asynchronous algorithm \eqref{eqn:fp_algo_dist_async} is analyzed next. First, a bounded delay is assumed. 

\begin{assumption} \label{asm:delay}
Define the worst-case communication delay as
\begin{equation}
T_d := \max_{i \in \cN, j \in \cN_i}\sup_{t \geq 1} \left \{t -  D^{(t)}_{i, j}\right\}
\end{equation}
and assume that $T_d$ is bounded; that is, $T_d < \infty$.
\end{assumption}

The following result holds. 

\begin{theorem} \label{thm:async_inf}
Suppose that the norm used is the $\ell_\infty$ norm. Then, under Assumptions \ref{asm:contr},  \ref{asm:sigma}, \ref{asm:approx}, and \ref{asm:delay},	the tracking error $\|\bx^{(t)} - \bx^{(*, t)} \|_\infty$ can be asymptotically bounded as:
\begin{equation}
\limsup_{t \rightarrow \infty} \|\bx^{(t)} - \bx^{(*, t)} \|_\infty \leq \frac{e_f + \sigma(1 + L T_d)}{1 - L}.
\end{equation}
\end{theorem}

\vspace{.2cm}

\begin{proof}
Notice first that the error introduced  by  the inexact mapping $\tilde{f}^{(t)}$ (rather than using $f^{(t)}$) can be bounded using steps similar to \eqref{eqn:inexact_bound}; therefore, we next focus on the case where the exact mapping is used. For any $i \in \cN$, let
\begin{equation}
\bz_i^{(t)} :=  \left (\bx_i^{(t)}, \{ \bx_j^{(D^{(t)}_{i,j})}\}_{j \in \cN_i} \right)
\end{equation}
denote the collection of variables at which the mapping of node $i$ is evaluated at time $t$. Then, using the $\ell_\infty$ norm, and noticing that $L_i \leq L$,  one has that
\begin{align}
&\|\bx^{(t)} - \bx^{(*, t-1)} \| = \max_{i \in \cN} \left \{ \left \|f^{(t-1)}_i (\bz_i^{(t-1)}) - \bx_i^{(*, t-1)}  \right\|  \right \} \nonumber \\
&=  \max_{i \in \cN} \left \{\left \|f^{(t-1)}_i (\bz_i^{(t-1)}) - f^{(t-1)}_i(\bz^{(*, t-1)})  \right\| \right\} \nonumber
\end{align}
\begin{align}
&\leq \max_{i \in \cN} L_i \Big [ \max \Big \{ \left\|\bx^{(t-1)}_i - \bx_i^{(*, t-1)} \right\|, \nonumber\\
& \qquad \qquad \Big \{\left\|\bx_j^{(D^{(t-1)}_{i,j})} - \bx_j^{(*, t-1)} \right \| \Big \}_{j \in \cN_i} \Big\} \Big ]\nonumber \\
&\leq L \max \Big \{\left\|\bx^{(t-1)} - \bx^{(*, t-1)} \right\|, \nonumber\\
& \qquad \qquad \max_{i \in \cN, j \in \cN_i}\Big \{\left\|\bx^{(D^{(t-1)}_{i,j})} - \bx^{(*, t-1)} \right \| \Big \}  \Big\}. \label{eqn:bound_async}
\end{align}
Next, the following inequalities can be obtained 
\begin{align}
& \left\|\bx^{(D^{(t-1)}_{i,j})} - \bx^{(*, t-1)} \right \|  \nonumber \\
& \leq \left\|\bx^{(D^{(t-1)}_{i,j})} - \bx^{(*, D^{(t-1)}_{i,j})} \right \| + \sum_{\ell = D^{(t-1)}_{i,j}}^{t-1} \| \bx^{(*, \ell)} - \bx^{(*, \ell-1)}\| \nonumber \\
& \leq \left\|\bx^{(D^{(t-1)}_{i,j})} - \bx^{(*, D^{(t-1)}_{i,j})} \right \| + \sigma T_d, \label{eqn:integral_sigma}
\end{align}
where the last inequality follows from Assumption \ref{asm:sigma} and Assumption \ref{asm:delay}. By plugging this last inequality into \eqref{eqn:bound_async}, one obtains
\begin{align*}
&\|\bx^{(t)} - \bx^{(*, t-1)} \| \leq L \max \Big \{\left\|\bx^{(t-1)} - \bx^{(*, t-1)} \right\|, \nonumber\\
& \qquad \qquad \max_{i \in \cN, j \in \cN_i}\Big \{\left\|\bx^{(D^{(t-1)}_{i,j})} - \bx^{(*, D^{(t-1)}_{i,j})} \right \| + \sigma T_d \Big \}  \Big\}, 
\end{align*}
which, in turn, yields
\begin{align*}
&\|\bx^{(t)} - \bx^{(*, t)} \| \leq \sigma + L \max \Big \{\left\|\bx^{(t-1)} - \bx^{(*, t-1)} \right\|, \nonumber\\
& \qquad \qquad \max_{i \in \cN, j \in \cN_i}\Big \{\left\|\bx^{(D^{(t-1)}_{i,j})} - \bx^{(*, D^{(t-1)}_{i,j})} \right \| + \sigma T_d \Big \}  \Big\}
\end{align*}
by using the triangle inequality and Assumption \ref{asm:sigma}. In other words, there exists 
$\delta^{(t)} \in \{1, \ldots, T_d \}$ such that
\begin{align*}
\|\bx^{(t)} - \bx^{(*, t)} \| &\leq \sigma + L \left (\|\bx^{(t - \delta^{(t)})} - \bx^{(*, t - \delta^{(t)})} \| +   \sigma T_d\right) \\
& = L \|\bx^{(t - \delta^{(t)})} - \bx^{(*, t - \delta^{(t)})} \| + \sigma (1 + L T_d),
\end{align*}
and the result follows by Lemma \ref{lem:geom} in the Appendix.
\end{proof}
The following result follows from the well-known equivalence of norms.
\begin{corollary} \label{cor:ell2}
Suppose that the norm used  is the $\ell_2$ norm, and that Assumptions \ref{asm:contr}, \ref{asm:sigma}, \ref{asm:approx}, and \ref{asm:delay} hold. 
If, in addition, $L < \frac{1}{\sqrt{m}}$, then
\begin{equation}
\limsup_{t \rightarrow \infty} \|\bx^{(t)} - \bx^{(*, t)} \|_\infty \leq \frac{e_f + \sigma(1 + L \sqrt{m} T_d)}{1 - L\sqrt{m}}.
\end{equation}
\end{corollary}

\vspace{.2cm}

\begin{proof}
Since $f^{(t)}$ is a contraction with respect to $\ell_2$ norm, one has that
\begin{align*}
\|f^{(t)}(\bx) - f^{(t)}(\bx')\|_\infty &\leq \|f^{(t)}(\bx) - f^{(t)}(\bx')\|_2 \\
&\leq L\|\bx -\bx'\|_2 \\
&\leq L\sqrt{m} \|\bx -\bx'\|_\infty.
\end{align*}
Therefore, $f^{(t)}$ is a contraction with respect to $\ell_\infty$ norm if $L < 1/\sqrt{m}$. The result then follows by Theorem \ref{thm:async_inf}.
\end{proof}

However, the bound of Corollary \ref{cor:ell2} represents a worst case and the bound is not necessarily tight. Therefore, a refined result is provided in the following.

\begin{theorem} \label{thm:ell2}
Let
\begin{equation}
N_d := \sup_{t \geq 1} \max_{i \in \cN} \sum_{j \in \cN_i} \mathbb{I} \{D_{i, j}^{(t)} < t \} \in [0, m-1]
\end{equation}
denote the maximum number of variables that are outdated at any given time step at any node. 
Suppose that the norm used  throughout is the $\ell_2$ norm, and that Assumptions \ref{asm:contr},  \ref{asm:sigma}, \ref{asm:approx}, and \ref{asm:delay} are satisfied. 
If, in addition, $L < \frac{1}{\sqrt{N_d + 1}}$, then
\begin{equation}
\limsup_{t \rightarrow \infty} \|\bx^{(t)} - \bx^{(*, t)} \|_2 \leq \frac{e_f + \sigma(1 + L \sqrt{N_d+1} T_d)}{1 - L\sqrt{N_d+1}}.
\end{equation}
\end{theorem}

\begin{proof}
Since $f^{(t)}$ is Lipschitz with respect to $\ell_2$ norm with coefficient $L$, there exist constants $\{L_i\}_{i \in \cN}$ such that for every $\bx, \bx' \in \cD$, we have that
\begin{equation}
\|f^{(t)}_i(\bx) - f^{(t)}_i(\bx')\| \leq  L_i \| \bx - \bx'\|
\end{equation}
such that $\sum_{i \in \cN} L_i^2 = L^2$.
Similarly to the proof of Theorem \ref{thm:async_inf}, we have that
\begin{align}
&\|\bx^{(t)} - \bx^{(*, t-1)} \|^2 = \sum_{i \in \cN} \left \{ \left \|f^{(t-1)}_i (\bz_i^{(t-1)}) - \bx_i^{(*, t-1)}  \right\|^2  \right \} \nonumber \\
&=  \sum_{i \in \cN} \left \{\left \|f^{(t-1)}_i (\bz_i^{(t-1)}) - f^{(t-1)}_i(\bx^{(*, t-1)})  \right\|^2 \right\} \nonumber\\
&\leq   \sum_{i \in \cN} L_i^2\Big [ \left\|\bx^{(t-1)}_i - \bx_i^{(*, t-1)} \right\|^2 \nonumber\\
& \qquad \qquad + \sum_{j \in \cN_i} \left\|\bx_j^{(D^{(t-1)}_{i,j})} - \bx_j^{(*, t-1)} \right \|^2  \Big ]\nonumber \\
&\leq   \left(\sum_{i \in \cN} L_i^2 \right ) \max_{i \in \cN}\Big [ \left\|\bx^{(t-1)}_i - \bx_i^{(*, t-1)} \right\|^2 \nonumber\\
& \qquad \qquad + \sum_{j \in \cN_i} \left\|\bx_j^{(D^{(t-1)}_{i,j})} - \bx_j^{(*, t-1)} \right \|^2  \Big ]\nonumber 
\end{align}
\begin{align}
&\leq  L^2 \Big ( \left\|\bx^{(t-1)} - \bx^{(*, t-1)} \right\|^2 \nonumber\\
& \qquad  + \max_{i \in \cN} \sum_{j \in \cN_i: \, N_{i, j}^{(t-1)} < t-1}\Big \{\left\|\bx^{(D^{(t-1)}_{i,j})} - \bx^{(*, t-1)} \right \|^2 \Big \}  \Big\} \nonumber\\
&\leq L^2 (N_d + 1) \max \Big \{  \left\|\bx^{(t-1)} - \bx^{(*, t-1)} \right\|^2, \nonumber\\
& \qquad  \max_{i \in \cN, j \in \cN_i:\, N_{i, j}^{(t-1)} < t-1}\Big \{\left\|\bx^{(D^{(t-1)}_{i,j})} - \bx^{(*, t-1)} \right \|^2 \Big \}   \Big \}.
\label{eqn:bound_async_2}
\end{align}
By taking the square root of both sides of \eqref{eqn:bound_async_2} and using \eqref{eqn:integral_sigma}, we obtain that
\begin{align}
&\|\bx^{(t)} - \bx^{(*, t-1)} \| \leq L \sqrt{N_d + 1} \max \Big \{  \left\|\bx^{(t-1)} - \bx^{(*, t-1)} \right\| \nonumber\\
& \hspace{-.2cm} \max_{i \in \cN, j \in \cN_i:\, N_{i, j}^{(t-1)} < t-1}\Big \{ \left\|\bx^{(D^{(t-1)}_{i,j})} - \bx^{(*, D^{(t-1)}_{i,j})} \right \| + \sigma T_d \Big \}   \Big \}. \hspace{-.2cm}
\label{eqn:bound_async_3}
\end{align}
Thus, there exists 
$\tilde{\delta}^{(t)} \in \{1, \ldots, T_d \}$ such that: 
\begin{align*}
&\|\bx^{(t)} - \bx^{(*, t-1)} \| \\
& \quad \leq L \sqrt{1 + N_d} \left ( \|\bx^{(t - \tilde{\delta}^{(t)})} - \bx^{(*, t - \tilde{\delta}^{(t)})} \| + \sigma T_d \right ), 
\end{align*}
and the proof is completed similarly to that of Theorem \ref{thm:async_inf}.
\end{proof}

\section{Examples}
\label{sec:examples}

\subsection{Feedback-based gradient  methods}
\label{sec:feedback}

Consider the following online projected gradient algorithm associated with a time-varying constrained optimization problem: 
\begin{equation}
\bx^{(t+1)} = \proj_{\cX^{(t)} \times \mathbb{R}^m_+} \left \{\bx^{(t)} - \alpha \bPhi^{(t)}(\bx^{(t)}) \right\}
\end{equation}
where $\bx^{(t)}$ is in this case a vector stacking the primal and dual variables; $\cX^{(t)}$ is a time-varying convex and compact set; $\alpha > 0$ is the step size; and $\bPhi^{(t)}$ is a strongly monotone and Lipschitz map associated with a regularized Lagrangian function (see~\cite{Koshal11,opfPursuit,Andreani11} for a detailed description). For example, if the problem features a constraint in the form $g_i^{(t)}(\bu) \leq 0$, with $\bu$ denoting the primal variables, then one entry of the map $\bPhi^{(t)}$ is $- g_i^{(t)}(\bu^{(t)})$. Suppose that the function $g_i^{(t)}(\bu^{(t)})$ represents a measurable quantity, and let $\hat{g}^{(t)}$ be a measurement of $g_i^{(t)}(\bu^{(t)})$ at time $t$. The main idea behind feedback-based online optimization methods is to replace the function $g_i^{(t)}(\bu^{(t)})$ with the measurement $\hat{g}^{(t)}$ in the map $\bPhi^{(t)}$~\cite{opfPursuit}. With this change, as well as by possibly replacing $\bu^{(t)}$ with its measurements, one can define an approximate map $\tilde{\bPhi}^{(t)}$, where the error  $\|\tilde{\bPhi}^{(t)}(\bx^{(t)}) - \bPhi^{(t)}(\bx^{(t)})\|$ introduced by the measurement noise and model mismatches can be bounded uniformly in time.  

To concretely outline an illustrative example while respecting space limitations, consider the simplified setting of a regularized projected gradient method with a \emph{strongly smooth} convex objective function $g^{(t)}$ with parameter $M$; namely, for all $\bx, \bx' \in \reals^m$, we have $\|\nabla g^{(t)} (\bx) - \nabla g^{(t)} (\bx')\| \leq M \|\bx -\bx' \|$. In this case,
\begin{equation}
\label{eq:gradient_example}
f^{(t)}(\bx) := \proj_{\cX^{(t)}} \left \{\bx - \alpha \left(\nabla g^{(t)} (\bx) + \eta \bx  \right) \right\},
\end{equation}
where $\eta > 0$ is a regularization parameter.
In this case, an approximate map would be $\tilde{f}^{(t)}(\bx) := \proj_{\cX^{(t)}} \left \{\bx - \alpha \hat{g}^{(t)}  \right\}$, with $\hat{g}^{(t)}$ an estimate or measurement of  $\nabla g^{(t)} (\bx^{(t)}) + \eta \bx^{(t)}$. It is well known that $f^{(t)}$ is Lipschitz with constant $L := \max \{|1 - \alpha \eta|, |1 - \alpha (M + \eta)| \}$ on $\reals^m$ in the $\ell_2$ norm; see, e.g., \cite{monotonePrimer}. To satisfy the condition of Theorem \ref{thm:ell2}, it is required that
\[
\max \{|1 - \alpha \eta|, |1 - \alpha (M + \eta)| \} < \frac{1}{\sqrt{N_d + 1}}.
\]
This yields the following set of conditions:
\begin{align*}
&1 - \alpha (M + \eta) \geq -\frac{1}{\sqrt{N_d + 1}} \\
&1 - \alpha \eta \leq \frac{1}{\sqrt{N_d + 1}} \\
\end{align*}
and these conditions are equivalent to
\begin{equation} \label{eqn:alpha_cond}
\frac{1}{\eta}\left(1 - \frac{1}{\sqrt{N_d + 1}}\right) \leq \alpha \leq \frac{1}{M + \eta}\left(1 + \frac{1}{\sqrt{N_d + 1}}\right).
\end{equation}
To make the set in \eqref{eqn:alpha_cond} nonempty, we require that
\begin{equation}
\frac{\sqrt{N_d + 1} - 1}{\sqrt{N_d + 1} + 1} := \kappa < \frac{\eta}{M + \eta} 
\end{equation}
which gives a lower bound on the regularization parameter as
\begin{equation} \label{eqn:eta_cond}
\eta > \frac{\kappa}{1 - \kappa} M = \frac{\sqrt{N_d + 1} - 1}{2} M.
\end{equation}
Note that, in the synchronous case, $N_d = 0$, and therefore condition \eqref{eqn:eta_cond} reduces to $\eta > 0$ as expected.

\subsection{Multi-Area Load Flow in Power Systems}
\label{sec:loadflow}

Consider the standard load-flow problem in a power network with $n$ constant-power buses and a single slack bus. It was recently shown in \cite{multiphaseArxiv} that this problem can be cast as the following fixed-point problem
\begin{equation} \label{eqn:pf_static}
\bv = h(\bv, \bs)
\end{equation}
where the vector $\bv \in \reals^{2n}$ collects the voltage phasors and the vector $\bs \in \reals^{2n}$ collects the active and reactive power injections across the buses. In particular, under certain conditions on the power injections $\bs$, the map $h(\cdot, \bs)$ is a contraction and self-map in the $\ell_\infty$ norm, on some (proper) subset $\cD$ of $\reals^{2n}$; see \cite{multiphaseArxiv} for details. 

Observe that problem \eqref{eqn:pf_static} is a static problem that presupposes a fixed power injection vector $\bs$. However, in modern power networks with high	 penetration of renewable energy sources and flexible loads, the power injections vary at a fast time scale. Therefore, consider a sequence $\{\bs^{(t)}\}_{t = 1}^\infty$, yielding the following time-varying problem
\begin{equation} \label{eqn:pf_dynamic}
\bv = h(\bv, \bs^{(t)}), \, t \in \mathbb{N}.
\end{equation}
Note that \eqref{eqn:pf_dynamic} fits the proposed framework, with $f^{(t)} (\cdot) := h(\cdot, \bs^{(t)})$, and the results of Theorem \ref{thm:result_fixedpoint} applies.

Consider now a multi-area problem, whereby the load-flow computation is distributed across $N$ physical areas of the power network. Due to space constraints, we present below the three-area case with a single connection point between them; see Fig.~\ref{fig:F_system} for an illustration.
Let $\bv_i$ and $\bs_i$ denote the vectors collecting the voltage phasors and power injections of area $i$, respectively. Suppose that area $1$ contains the slack bus, and let $\bv_{12} \in \reals^2$ be the element of $\bv_1$ representing the voltage phasor at the connection point between area $1$ and $2$, and $\bv_{23}$ be the element of $\bv_2$ representing the voltage phasor at the connection point between area $2$ and $3$. In this case, the problem \eqref{eqn:pf_dynamic} can be decomposed as 
\begin{subequations} \label{eqn:pf}
\begin{align}
\bv_1 &= h_1(\bv_1, \bs_1^{(t)}, g_2(\bv_2, \bv_{12})) \\
\bv_2 &= h_2(\bv_2, \bv_{12}, \bs_2^{(t)}, g_3(\bv_3, \bv_{23})) \\
\bv_3 &= h_2(\bv_3, \bv_{23}, \bs_3^{(t)}),
\end{align}
\end{subequations}
where $h_i$ represent the power flow mapping of area $i$, and $g_j(\bv_j, \bv_{ij})$ is the power injection at the connection point between area $i$ and $j$ computed from the knowledge of the voltages in area $j$ and the voltage at the connection point. In practice, the value of $g_j(\bv_j, \bv_{ij})$ can be measured in area $i$ similarly to the rest of power injections in area $i$; let $\bs_{ji}$ denote the measured value.
The value of $\bv_{ij}$ can be either communicated or measured with a phasor measurement unit. 
This leads to the following \emph{inexact} load-flow mappings
\begin{subequations} \label{eqn:pf_online}
\begin{align}
\tilde{f}_1^{(t)} (\bv_1) &:= h_1(\bv_1, \bs_1^{(t)},\bs_{21}^{(t)})  \\
\tilde{f}_2^{(t)} (\bv_1, \bv_2 ) &:= h_2(\bv_2, \bv_{12}, \bs_2^{(t)}, \bs_{32}^{(t)}) \\
\tilde{f}_3^{(t)} (\bv_2, \bv_3 ) &:= h_3(\bv_3, \bv_{23}, \bs_3^{(t)}),
\end{align}
\end{subequations}
where $\bs^{(t)}_{ji}$ is the measurement of $g_j(\bv_j^{(t, *)}, \bv_{ij}^{(t, *)})$,
and the result of Theorem \ref{thm:async_inf} applies. We call algorithm \eqref{eqn:fp_algo_dist_async} that is based on the mapping \eqref{eqn:pf_online} a \emph{feedback-based} load flow, as it relies on the measurement of $g_j(\bv_j^{(t, *)}, \bv_{ij}^{(t, *)})$ rather than on the explicit evaluation of $g_j(\bv_j^{(t)}, \bv_{ij}^{(t)})$ as prescribed by the exact mapping defined in \eqref{eqn:pf}.

\section{Illustrative Numerical Results}
\label{sec:results}

\subsection{Feedback-based gradient method}

Following the example of Section \ref{sec:feedback}, consider the time-varying problem 
\begin{align}
\min_{\{x_i \in [x_i^m, x_i^M]\}_{i = 1}^N} h^{(t)}(\bx) + \frac{\gamma}{2} (y^{(t)}(\bx) - r^{(t)})^2 + \frac{\eta}{2}\|\bx\|_2^2 \label{eq:example_problem}
\end{align}
where: $h^{(t)}(\bx) := \sum_{i = 1}^N h_i^{(t)}(x_i)$, $\{h_i^{(t)}\}$ are convex functions; $\gamma > 0$ is a given weight; $y^{(t)}(\bx) := \bc^\sfT \bx + w^{(t)}$ is a synthetic model for a measurable quantity $y^{(t)}$;  $r^{(t)}$ is a given reference for $y^{(t)}(\bx)$; $\eta > 0$ is a regularization parameter; and, $x_i^m$, $x_i^M$ are known parameters (with $x_i^m < x_i^M$). If the function $h^{(t)}(\bx)$ is strongly convex, then the parameter $\eta$ can be set to $0$. The problem~\eqref{eq:example_problem} seeks a trade off between tracking of a reference $r^{(t)}$ associated with $y^{(t)}(\bx)$ and costs associated with some control actions $\bx$.    

Similar to, e.g.,~\eqref{eq:online}, assuming that $h^{(t)}(\bx)$ is strongly convex and $\eta = 0$, an online gradient descent method involves a sequential execution of the following step:
\begin{align}
 \bx^{(t+1)} & = \proj_{\cX} \Big\{ \bx^{(t)} - \alpha \big(\nabla_\bx h^{(t)}(\bx^{(t)})  \nonumber \\
 & \hspace{2.2cm} + \gamma \bc(y^{(t)}(\bx^{(t)}) - r^{(t)} ) \big) \Big\}
\end{align}
where $\cX$ is the Cartesian product of the sets $\{[x_i^m, x_i^M]\}$. On the other hand, a feedback-based gradient descent method amounts to: 
\begin{align}
 \bx^{(t+1)} & = \proj_{\cX} \Big\{ \bx^{(t)} - \alpha \big(\nabla_\bx h^{(t)}(\bx^{(t)})  \nonumber \\
 & \hspace{2.2cm} + \gamma \bc(\hat{y}^{(t)} - r^{(t)} ) \big) \Big\} \label{eq:gradient_example_problem}
\end{align}
with $\hat{y}^{(t)} $ a measurement of $y^{(t)}(\bx^{(t)})$ at time $t$. Notice that the computation of~\eqref{eq:gradient_example_problem} can be decoupled across the entries of the vector $\bx$; in fact, upon receiving $\hat{y}^{(t)} - r^{(t)}$ from a central entity, each agent $i = 1,\ldots N$ can update $x_i$ as $x_i^{(t+1)} = \proj_{[x_i^m, x_i^M]} \{ x_i^{(t)} - \alpha (\nabla_{x_i} h_i^{(t)}(x_i^{(t)}) + \gamma c_i(\hat{y}^{(t)} - r^{(t)} ) ) \}$ (the communication network has, in this case, a star topology).

\begin{figure}[t!]
  \centering
  \includegraphics[width=1.05\columnwidth]{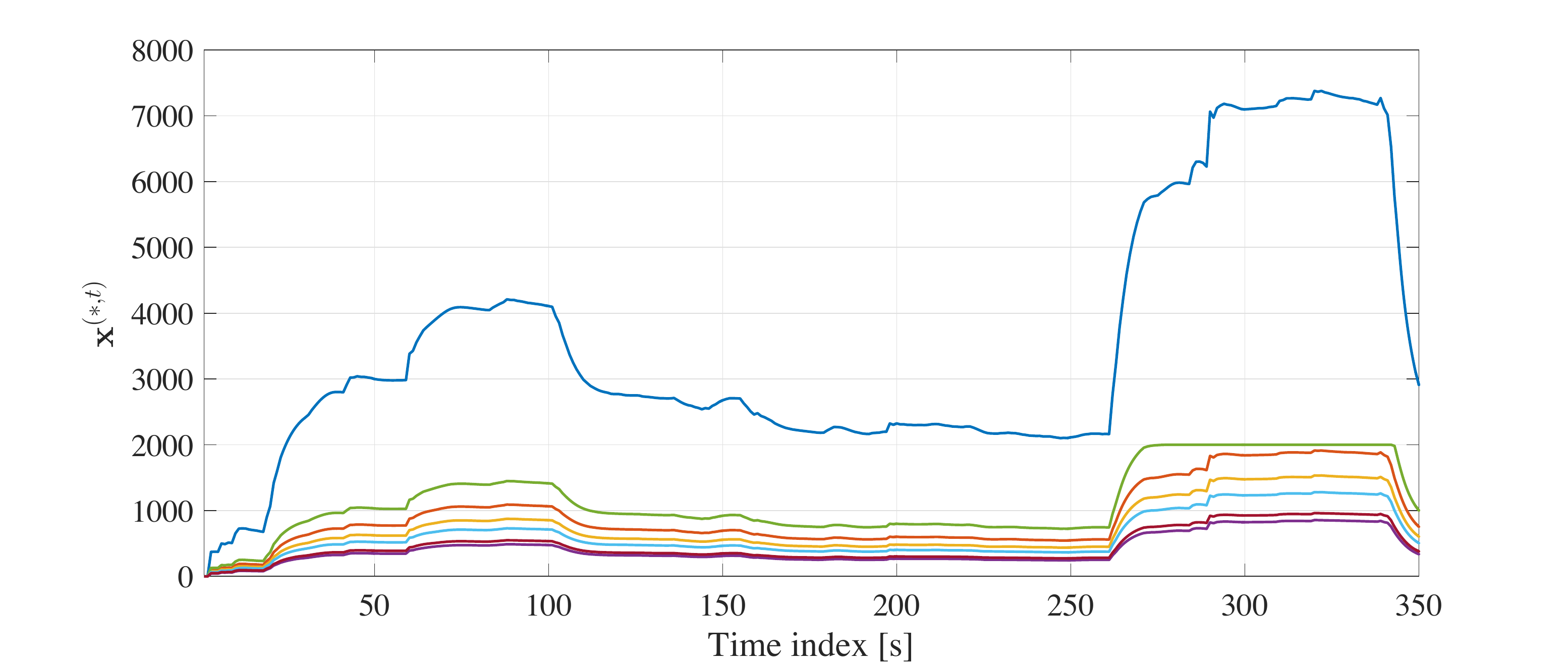}
\vspace{-.4cm}
\caption{Evolution of the optimal solution of problem~\eqref{eq:example_problem}.} \label{fig:F_optimal_solution}
\vspace{-.5cm}
\end{figure}
\begin{figure}[t!]
  \centering
  \includegraphics[width=1.0\columnwidth]{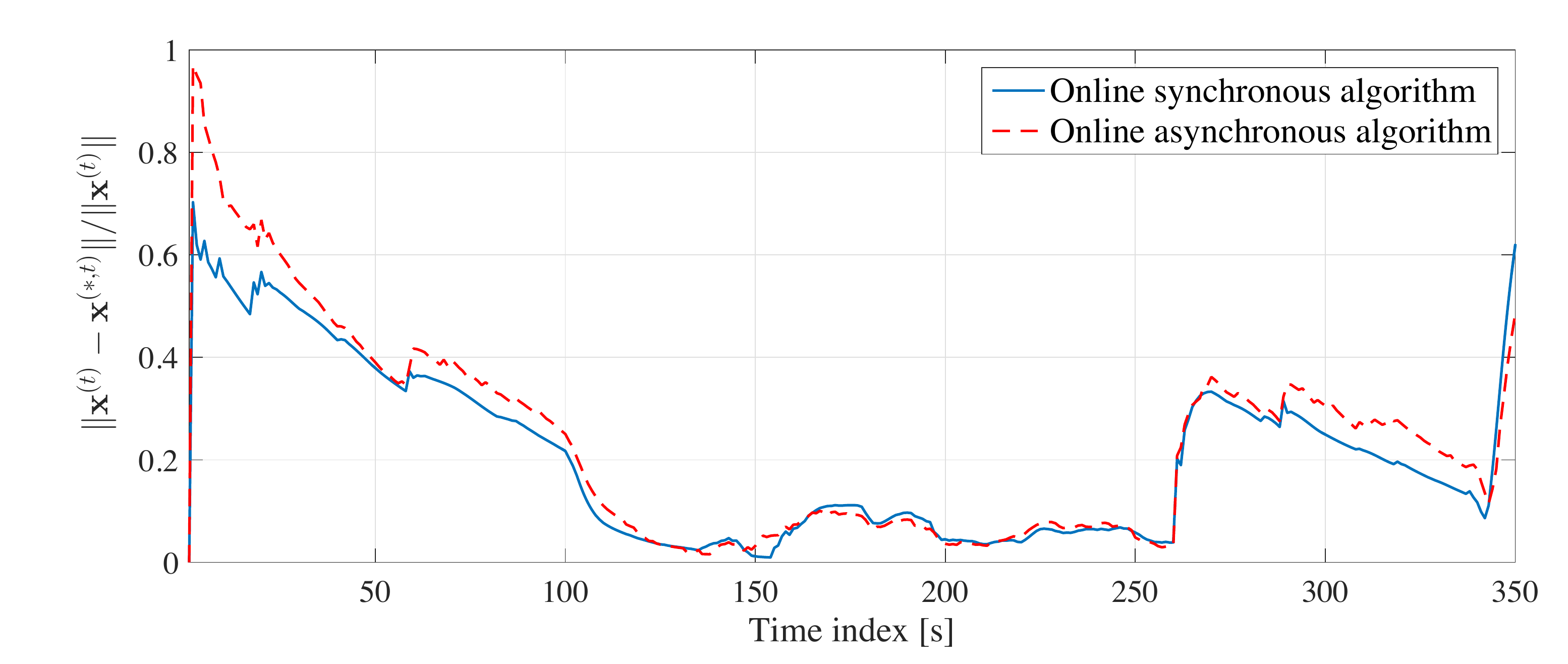}
\vspace{-.4cm}
\caption{Tracking error for the online projected gradient algorithm.} \label{fig:F_error_gradient_method}
\vspace{-.5cm}
\end{figure}

We apply~\eqref{eq:example_problem} and~\eqref{eq:gradient_example_problem} to the feeder in Fig.~\ref{fig:F_system}, where $\bx$ represents a vector of powers of seven controllable devices, and non-controllable powers vary at every second. A linear power flow model~\cite{multiphaseArxiv} is utilized to build the vector $\bc$, and the cost functions are set as $h_i^{(t)}(x_i) = (a_i/2) x_i^2$ with $a_i > 0$. 

Fig.~\ref{fig:F_optimal_solution} illustrates the evolution of the optimal vector $\bx^{(\ast,t)}$ when problem~\eqref{eq:example_problem} is solved to convergence at each time step. On the other hand, Fig.~\ref{fig:F_error_gradient_method} shows the tracking error $\|\bx^{(t)} - \bx^{(\ast, t)}\|_2 / \|\bx^{(\ast, t)}\|_2$ for: a synchronous version of the algorithm~\eqref{eq:gradient_example_problem}, with one step computed at every second; and, an asynchronous version where the packet containing $\hat{y}^{(t)} - r^{(t)}$ can be lost with probability $0.1$. The packet drop probability is assumed to be identical for each communication link. Although Fig.~\ref{fig:F_error_gradient_method} provides an illustrative snapshot, the tracking errors of the algorithms are bounded; the error increases when the optimal solution $\bx^{(\ast,t)}$ changes fast with a somewhat extreme scenario after $250$ seconds), but then settles to lower values as soon as the evolution of the optimal solution is smoother. The asynchronous algorithm exhibits a higher error, as expected from the analytical bound.

\subsection{Multi-Area Load Flow}

To test the online asynchronous load-flow method described in Section~\ref{sec:loadflow}, consider the IEEE 37-node test feeder illustrated in Fig.~\ref{fig:F_system}. The feeder is three-phase and unbalanced, and it has a nominal voltage of $4800$ V. For this test, it is partitioned into three areas. Algorithms that are based on~\eqref{eqn:pf_online} are tested next. 

Fig.~\ref{fig:F_system} illustrates the convergence of the algorithm for the static load flow (that is, where the power injections are fixed) for three cases: the synchronous algorithm; the feedback-based synchronous algorithm; and the feedback-based asynchronous algorithm, where the probability of a packet drop on the communication links between the areas is set to $0.5$. The plot confirms that the algorithm converges, with the feedback-based synchronous setting yielding the best convergence rate.     

Fig.~\ref{fig:F_tracking_fixedpoints} illustrates the tracking error when the vector of powers $\bs^{(t)}$ is time varying; in particular, the time interval is of $1$ second, and powers are varying at each time instant. The error $\|\bv^{(t)} - \bv^{(\ast, t)}\|_2 / \|\bv^{(\ast, t)}\|_2$, where $\bv^{(\ast, t)}$ is the exact fixed point and $\bv^{(t)}$ is the iterate produced by the algorithm is plotted for: the online synchronous algorithm; the online asynchronous algorithm, with a packet drop probability of $0.01$ on each communication link; and, the asynchronous scheme with packet drop probability of $0.1$. It can be seen that packet drops increase the tracking error, but the error remains bounded. Finally, Fig.~\ref{fig:F_voltages_tracking} illustrates the evolution of the true and estimated voltages at two representative connection points of the network (magnified for improve readability). It can be seen that, in spite of packet drops, the voltage error is negligible. 

\begin{figure}[t!]
  \centering
  \includegraphics[width=1.0\columnwidth]{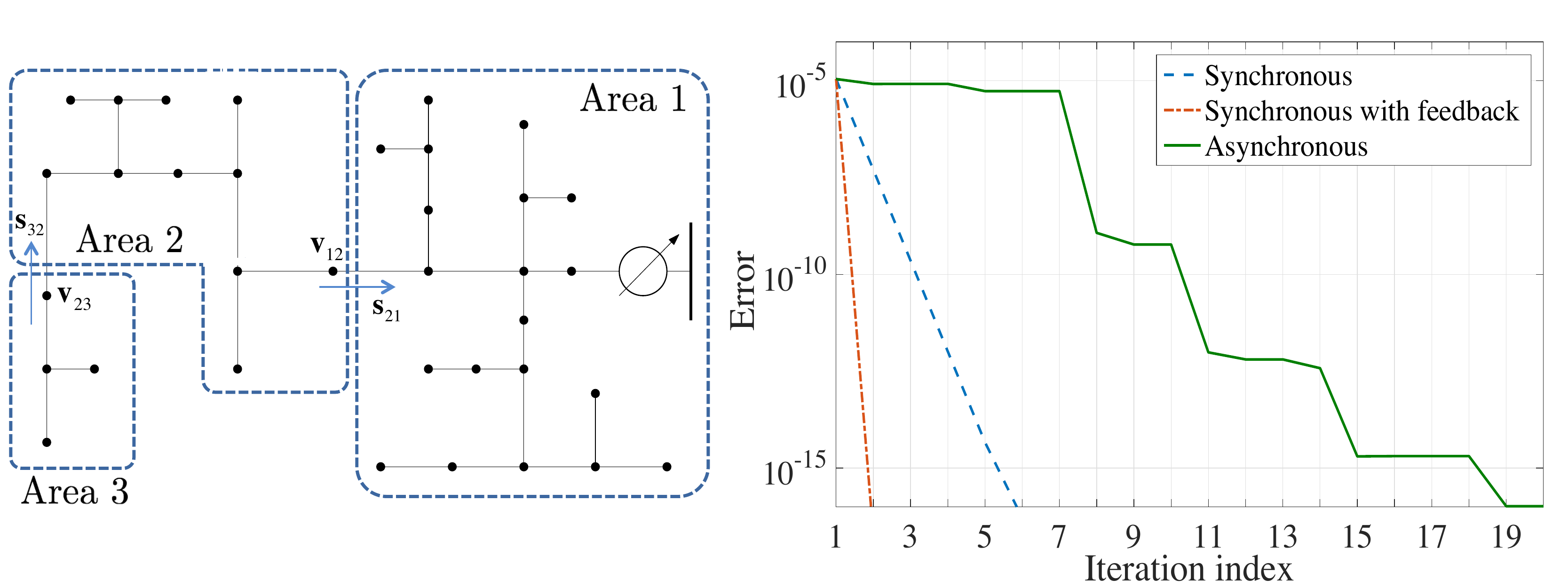}
\vspace{-.4cm}
\caption{Test system (left); convergence of the synchronous and asynchronous algorithms for the load flow in the time invariant case (right).} \label{fig:F_system}
\vspace{-.5cm}
\end{figure}
\begin{figure}[t!]
  \centering
  \includegraphics[width=1.0\columnwidth]{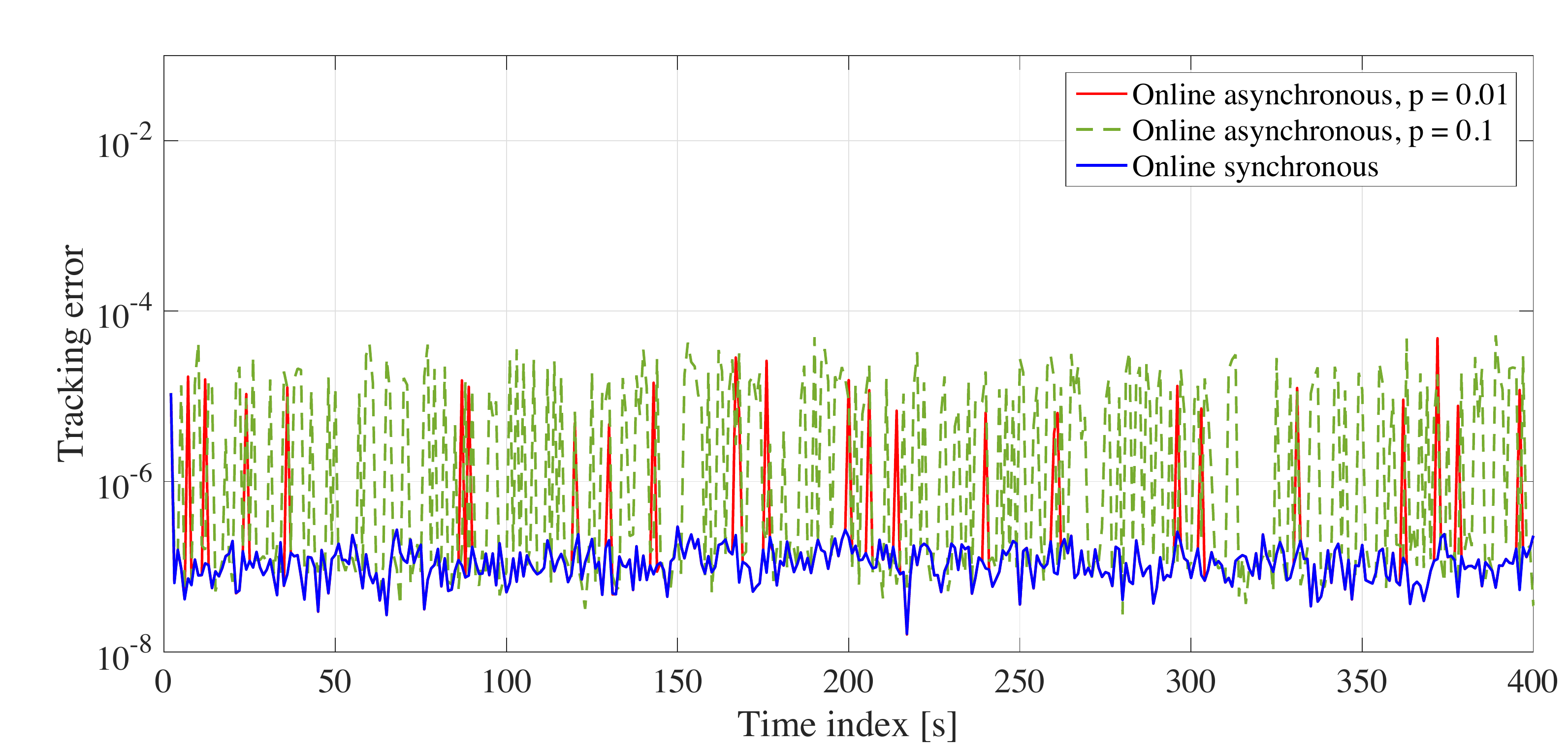}
\vspace{-.4cm}
\caption{Tracking error $\|\bv^{(t)} - \bv^{(\ast, t)}\|_2 / \|\bv^{(\ast, t)}\|_2$.} \label{fig:F_tracking_fixedpoints}
\end{figure}
\begin{figure}[t!]
  \centering
  \hspace{-.3cm} \includegraphics[width=1.04\columnwidth]{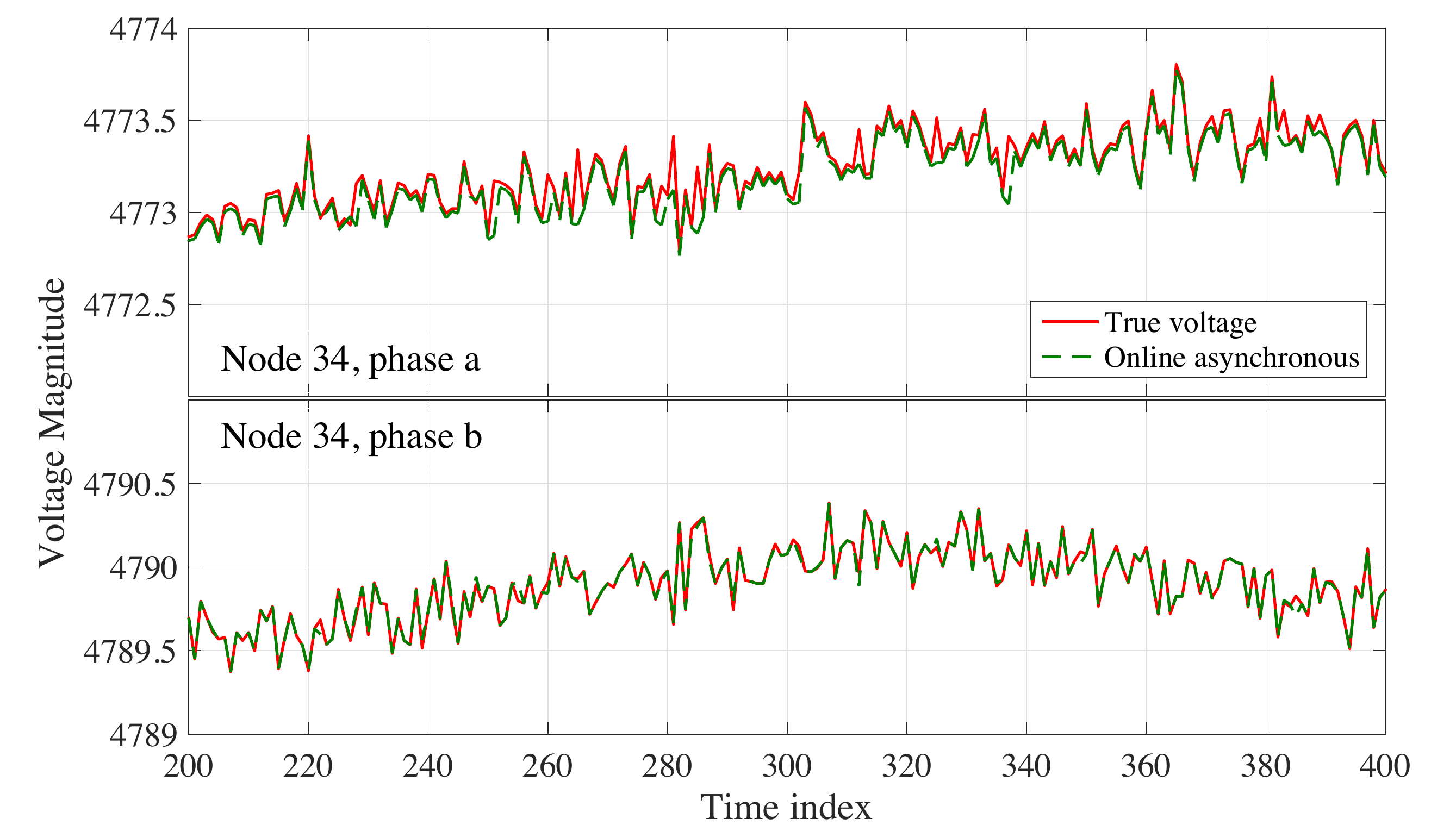}
\vspace{-.4cm}
\caption{Evolution of the voltages for two connection points. The probability of a packet loss is $0.1$ per communication link. } \label{fig:F_voltages_tracking}
\vspace{-.5cm}
\end{figure}

\section{Conclusion}
\label{sec:conc}
This paper developed an algorithmic framework for tracking fixed points of time-varying contraction mappings. Analytical results for the tracking error were established for the cases where only an imperfect information of the map is available; and the algorithm is implemented in a distributed fashion, with communication delays and packet drops leading to asynchronous algorithmic updates. 
Application to several classes of problems, including gradient-based methods for time-varying convex programs and multi-area load-flow, was demonstrated, and illustrative numerical results were provided. Future research directions include the extension of the results to a larger class of mappings, and to a more general asynchronous setting with non-homogeneous update rates.

\appendix

\begin{lemma} \label{lem:geom}
Let $\{a^{(t)}\}$ be a given positive sequence. Suppose that there exist $T < \infty$, $b < \infty$,  and $0 < \Gamma < 1$ such that, for  all $t > T$:
\[
a^{(t)} \leq b + \Gamma a^{(t - \delta^{(t)})},
\]
for some $ \delta^{(t)} \in \{1, \ldots, T\}$. Then,
\[
\limsup_{t \rightarrow \infty} a^{(t)} \leq b (1 - \Gamma)^{-1}.
\]
\end{lemma}

\vspace{.2cm}

\begin{proof}
By the hypothesis of the lemma, for each $t$, there exist $N(t) \in \{ \lfloor \frac{t-1}{T} \rfloor, \ldots, t-1 \}$ and $t_0(t) \leq T$, such that

\[
a^{(t)} \leq b \left (\sum_{i=0}^{N(t)-1} \Gamma^i \right) + \Gamma^{N(t)} a^{\left(t_0(t) \right)}.
\]
The result follows by noticing that $N(t) \rightarrow \infty$ as $t \rightarrow \infty$ and $t_0(t)$ is bounded for all $t$.

\end{proof}

\bibliographystyle{IEEEtran}
\bibliography{biblio.bib}

\end{document}